\documentclass[a4paper,12pt,twoside]{amsart}

\usepackage[cm]{fullpage}

\usepackage{mymacros,oitp}

\title{The ring of modular forms
for the even unimodular lattice of signature (2,10)}
\author[K.~Hashimoto]{Kenji Hashimoto}
\address{
Graduate School of Mathematical Sciences,
The University of Tokyo,
3-8-1 Komaba,
Meguro-ku,
Tokyo,
153-8914,
Japan.}
\email{hashi@ms.u-tokyo.ac.jp}
\author[K.~Ueda]{Kazushi Ueda}
\address{
Graduate School of Mathematical Sciences,
The University of Tokyo,
3-8-1 Komaba,
Meguro-ku,
Tokyo,
153-8914,
Japan.}
\email{kazushi@ms.u-tokyo.ac.jp}
\date{}
\begin{document}

\maketitle

\begin{abstract}
We show that the ring of modular forms with characters
for the even unimodular lattice of signature (2,10)
is generated by forms of weights
4, 10, 12, 16, 18, 22, 24, 28, 30, 36, 42, and 252
with one relation of weight 504.
The proof is based on the comparison
of the orbifold quotient
of the symmetric domain
with the root stack of the coarse moduli space.
\end{abstract}

\section{Introduction}
 \label{sc:introduction}

Let $P$ be an even non-degenerate lattice
of signature $(1, t)$ for some $0 \le t \le 19$.
A {\em $P$-polarized K3 surface} is a pair
$(Y, j)$ of a K3 surface $Y$
and a primitive lattice embedding
$
 j \colon P \hookrightarrow \Pic Y.
$
%The global Torelli theorem for K3 surfaces
%allow us to describe
%the moduli space $M^\circ$ of ample $M$-polarized K3 surfaces
%as the quotient $\cD^\circ / \Gamma$
%of an open subset of a bounded domain of type IV
%by a discrete group $\Gamma$.
%
%Choose an isotropic primitive vector $f$
%in the orthogonal complement $\bsM^\bot$
%of $\bsM$ in $L$, and
%consider the lattice 
%$
% \check{\bsM} := (\bZ f)^\bot_{\bsM^\bot} / \bZ f
%$
Lattice polarized K3 surfaces are introduced
by Nikulin \cite{MR544937}
and used by Dolgachev \cite{MR1420220}
to study mirror symmetry for K3 surfaces.
The \emph{mirror moduli space} of $P$-polarized K3 surfaces
is the moduli space of $\Pv$-polarized K3 surfaces,
where $\Pv \coloneqq (P \bot U)^\bot$ is the orthogonal complement
of the orthogonal sum of $P$ and
the even unimodular hyperbolic lattice $U$ of rank 2
inside the K3 lattice
$L \coloneqq E_8 \bot E_8 \bot U \bot U \bot U$.
Here we consider $P$ as a primitive sublattice of $L$ by the embedding $j$.

%The lattice of type $T_{2, 3, 7}$
\begin{figure}[t]
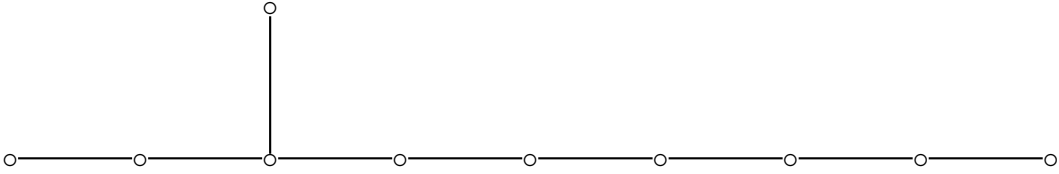

$$
\begin{psmatrix}
 & & \circ \\
 \circ & \circ & \circ & 
 \circ & \circ & \circ & 
 \circ & \circ & \circ
\end{psmatrix}
\ncline{2,1}{2,2}
\ncline{2,2}{2,3}
\ncline{2,3}{2,4}
\ncline{2,4}{2,5}
\ncline{2,5}{2,6}
\ncline{2,6}{2,7}
\ncline{2,7}{2,8}
\ncline{2,8}{2,9}
\ncline{1,3}{2,3}
$$
\caption{The Coxeter--Dynkin diagram
of the lattice $T_{2, 3, 7} = E_8 \bot U$}
\label{fg:237-diagram}
\end{figure}

Let $T_{2,3,7}$ be the lattice
determined by the Coxeter--Dynkin diagram
shown in \pref{fg:237-diagram}.
This lattice is isomorphic to $E_8 \bot U$ as an abstract lattice.
This is the most symmetric lattice
from the point of view of mirror symmetry,
in the sense that the mirror dual lattice $\Tv_{2,3,7}$ is isomorphic
to the original one;
$
 L \cong T_{2,3,7} \bot T_{2,3,7} \bot U.
$
Let $M$ be the coarse moduli space of
%pseudo-ample
$T_{2,3,7}$-polarized K3 surfaces,
which is isomorphic to the quotient
$\cD/\Gamma$
of a symmetric domain $\cD$ of type IV
by a discrete group $\Gamma$
 (see Section \ref{sc:lattice-polarization} for the definition of $\cD$ and $\Gamma$).
%Building on the work of Pinkham
%\cite{MR0498543}
Looijenga
\cite{MR761312}
proved that the graded ring of modular forms
(without characters)
is a polynomial ring
generated in degrees
\begin{align}
 \bw \coloneqq (4,10,12,16,18,22,24,28,30,36,42),
\end{align}
and
Brieskorn
\cite[Theorem 5]{MR642697}
proved that
the period map induces an isomorphism
from  the weighted projective space
$T^* \coloneqq \bfP(\bw)$ of weight $\bw$
to the Satake--Baily--Borel compactification
$M^*$ of $M$.

%\begin{theorem}[{Shiga \cite{Shiga_K3MIII}}]
% \label{th:shiga}
%There is an isomorphism
%$T^* \simto M^*$
%of algebraic varieties.
%%which sends the boundary $M^* \setminus M$
%%to 
%%There is a family $\frakY \to T^*$
%%of anticanonical hypersurfaces in $\bfP(1,6,14,21)$
%%over $T^*$
%%%the weighted projective space
%%%$
%%% T^* = \bfP(\bw)
%%%$
%%such that
%%\begin{enumerate}
%% \item
%%there is a subscheme $T \subset T^*$
%%satisfying $T^* \setminus T \cong \bfP^1$,
%% \item
%%geometric points of $T$ are
%%in one-to-one correspondence
%%with isomorphism classes of $T_{2,3,7}$-polarized K3 surfaces, and
%% \item
%%the period map
%%$
%% \Pi \colon T \to M
%%$
%%extends to an isomorphism
%%$
%% \Pi^* \colon T \simto M^*.
%%$
%%\end{enumerate}
%\end{theorem}

Let $\bM \coloneqq [\cD/\Gamma]$ and
$\bP(\bw) \coloneqq [(\bC^{11} \setminus \bszero)/\bCx]$
be the orbifold quotients,
whose coarse moduli spaces are
$M$ and $T^*$ respectively.
For a pair of an orbifold and a divisor on it,
one can perform the \emph{root construction}
\cite{MR2450211,Cadman_US}
to introduce a generic stabilizer along the divisor.
%There is a divisor $\bH^* \subset \bP(\bw)$ of weight 504,
%whose defining equation is given by the ratio
%$
% \Delta = d/r^3
%$
%of a discriminant $d$ and the cubic power $r^3$ of a resultant.
Let $\bT^*$ be the orbifold obtained from $\bP(\bsw)$
by the root construction of order two along a divisor $H_\bP$
of degree 504.
The main result of this paper is the following:

\begin{theorem} \label{th:main}
There is a bimeromorphic map
$\bT^* \dashrightarrow \bM$ of orbifolds,
which is an isomorphism in codimension one.
\end{theorem}

%This is a refinement of \pref{th:shiga},
%whose proof follows that of Shiga \cite{Shiga_K3MIII} closely.
The map in \pref{th:main} is not an isomorphism,
since the orbifold $\bT^*$ has singularities
coming from the singularities of the divisor $H_\bP$,
whereas $\bM$ is smooth (as an orbifold),
and $\bT^*$ is proper,
whereas $\bM$ is not.
%The orbifolds $\bT^*$ and $\bM$ are not isomorphic,
%not only because $\bM$ is non-compact,
%but also because $\bM$ is smooth whereas $\bT^*$ is singular.
%
Nevertheless,
as a corollary to \pref{th:main},
one obtains the following structure theorem
of the ring of modular forms:

\begin{corollary} \label{cr:main}
%The group $\Hom(\Gamma, \bCx)$ of characters of $\Gamma$
%is isomorphic to $\bZ/2 \bZ$ and generated by $\det$.
The graded ring of modular forms with characters for $\Gamma$
is generated
by forms of weights
4, 10, 12, 16, 18, 22, 24, 28, 30, 36, 42, and 252
with one relation of weight 504.
%Among these generators,
%only the form with weight $252$ is associated
%with a non-trivial character.
\end{corollary}

The group of characters of $\Gamma$ is isomorphic to $\mathbb{Z}/2\mathbb{Z}$
 and the relation in \pref{cr:main} comes from the defining equation
%\pref{eq:disc}
of the divisor $H_\bP \subset \bP(\bw)$.
%We expect that if one expands these discriminant and the resultant,
%then the resulting formula will be very complicated.
An explicit description of the generators
in terms of Eisenstein series and a Borcherds product
is given in \cite{1707.05029}.

This paper is organized as follows:
We collect basic definitions
on moduli spaces of lattice polarized K3 surfaces
and modular forms
in \pref{sc:lattice-polarization}.
In \pref{sc:shiga},
we recall the description of the coarse moduli space
of $T_{2,3,7}$-polarized K3 surfaces,
which is a particular case
of a much more general result of Looijenga
\cite{MR761312}.
%for the readers' convenience.
%everything in this section can be found
%in \cite{MR761312}
%in more general form.
The exposition in this section follows Shiga \cite{Shiga_K3MIII} closely.
%in this section.
In \pref{sc:H},
we lift the period map
$\Pi \colon T \to M$
between coarse moduli spaces
to a meromorphic map of orbifolds
%defined outside a complement of a subspace of codimension greater than one,
and prove \pref{th:main}.
We prove \pref{cr:main} in \pref{sc:cor},
and give a description
%\pref{eq:disc}
of the defining equation
of the divisor $H_\bP$
in terms of a discriminant and a resultant
in \pref{sc:disc}.

\emph{Acknowledgement}:
Special thanks go to Hironori Shiga
for kindly sharing the preprint \cite{Shiga_K3MIII}
with us.
We also thank Atsuhira Nagano for valuable discussions, and
Igor Dolgachev for pointing out the work of Looijenga.
K.~H.~is partially supported by Grants-in-Aid for Scientific Research
(17K14156).
K.~U.~is partially supported by Grant-in-Aid for Scientific Research
(15KT0105, 16K13743, 16H03930).

\section{Lattice polarized K3 surfaces}
 \label{sc:lattice-polarization}

Let
$
 L \coloneqq E_8 \bot E_8 \bot U \bot U \bot U
$
be the K3 lattice,
and $P \coloneqq T_{2,3,7}$ be the even unimodular lattice of signature $(1, 9)$
appearing in Introduction.
%Fix a connected component $V(P)^+$ of
%$
% V(P) = \lc x \in P_\bR \relmid (x,x)>0 \rc
%$
%and
We can choose a subset $\Delta(P)^+$ of
$
 \Delta(P) \coloneqq \lc \delta \in P \relmid (\delta, \delta) = -2 \rc
$
satisfying
\begin{enumerate}
 \item
$\Delta(P) = \Delta(P)^+ \coprod (- \Delta(P)^+)$ and
 \item
if $\delta_1,\delta_2 \in \Delta(P)^+$ and $\delta_1+\delta_2 \in \Delta(P)$,
 then $\delta_1+\delta_2 \in \Delta(P)^+$.
%$\Delta(P)^+$ is closed under addition (but not subtraction).
\end{enumerate}
There are ten indecomposable elements $\delta \in \Delta(P)^+$
(that is, $\delta$ cannot be written as $\delta_1+\delta_2$ with
$\delta_1,\delta_2 \in \Delta(P)^+$)
corresponding to the ten vertices in \pref{fg:237-diagram}.
The choice of $\Delta(P)^+$ is unique
up to the action of the orthogonal group $O(P)$
of the lattice $P$, and $O(P)$ is generated by the reflections along
 $\delta \in \Delta(P)^+$.
Define
\begin{align}
 C(P) &\coloneqq \lc h \in P \relmid
  (h, \delta) \ge 0 \text{ for any } \delta \in \Delta(P)^+ \rc, \\
 C(P)^\circ &\coloneqq \lc h \in P \relmid
  (h, \delta) > 0 \text{ for any } \delta \in \Delta(P)^+ \rc
\end{align}
and set
\begin{align}
 \Pic(Y)^{+} &\coloneqq C(Y) \cap H^2(Y; \bZ), \\
 \Pic(Y)^{++} &\coloneqq C(Y)^\circ \cap H^2(Y; \bZ),
\end{align}
where $C(Y)^\circ \subset H^{1,1}(Y) \cap H^2(Y; \bR)$
is the K\"{a}hler cone of $Y$
and $C(Y)$ is its closure.

\begin{definition}[{Nikulin \cite{MR544937}}]
%Let $P$ be an even non-degenerate lattice
%of signature $(1, t)$
%where $0 \le t \le 19$.
A {\em $P$-polarized K3 surface} is a pair $(Y, j)$
where $Y$ is a K3 surface and
$
 j \colon P \hookrightarrow \Pic(Y)
$
is a primitive lattice embedding.
An {\em isomorphism} of $P$-polarized K3 surfaces
$(Y, j)$ and $(Y', j')$ is an isomorphism
$f : Y \to Y'$ of K3 surfaces
such that $j = f^* \circ j'$.
%making the diagram
%\begin{align}
%\begin{psmatrix}[colsep=2cm]
%  & P & \\
% \Pic Y' & & \Pic Y
%\end{psmatrix}
%\psset{hookwidth=3mm,nodesep=2mm,shortput=nab,arrows=->}
%\ncline[arrows=H->,hookwidth=3mm]{1,2}{2,1}_{j'}
%\ncline[arrows=H->,hookwidth=-3mm]{1,2}{2,3}^{j}
%\ncline{2,1}{2,3}^{f^*}
% \label{eq:diagram1}
%\end{align}
%commute.
A $P$-polarized K3 surface is
\emph{pseudo-ample} if
$
 j(C(P)^\circ) \cap \Pic(Y)^+ \ne \emptyset,
$
and
\emph{ample} if
$
 j(C(P)^\circ) \cap \Pic(Y)^{++} \ne \emptyset.
$
\end{definition}

Fix a primitive lattice embedding
$
 i_P \colon P \hookrightarrow L
$
and let
$
 Q \coloneqq P^\bot
$
be the orthogonal complement inside $L$,
which is isomorphic to $P \bot U$
as an abstract lattice.
The period domain $\cD$ is
a connected component of
$$
 \{ [\Omega] \in \bfP(Q \otimes \bC)
  \mid (\Omega, \Omega) = 0, \ (\Omega, \Omegabar) > 0 \},
$$
which
%can be identified with the symmetric homogeneous space
%$
% O(2, 19-t) / SO(2) \times O(19-t)
%$
%of oriented positive-definite 2-planes in $Q_\bR$.
%It consists of two onnected components,
%each of which is isomorphic to
is a bounded Hermitian domain of type IV.
%Set
%\begin{align}
% \Gamma(P) = \{ \sigma \in O(L) \mid \sigma(m) = m
%  \text{ for any } m \in P \}
%\end{align}
%and $\Gamma$ be its image
%under the natural injective homomorphism
%$
% \Gamma(P) \hookrightarrow O(Q).
%$
%There is a subgroup $\Gamma^+$ of index two
%which preserves the connected component $\cD^+ \subset \cD$.
The global Torelli theorem \cite{MR0284440, MR0447635}
and the surjectivity of the period map \cite{MR592693}
show that the coarse moduli space of
pseudo-ample $P$-polarized K3 surfaces
is given by $M \coloneqq \cD/\Gamma$,
where $\Gamma \coloneqq O(Q)^+$
is the index two subgroup
of the orthogonal group of the lattice $Q$
preserving the connected component $\cD$,
which acts naturally on $\cD$
%Note that the action of $\Gamma$ on $\cD$ factors
through
%the quotient
$
 \Gammabar
  \coloneqq PO(Q)^+
   = \Gamma / \{ \pm \id \}.
$
The coarse moduli space of ample $P$-polarized K3 surfaces
is the subspace
$
 \left. \lb \cD \setminus H_\cD \rb \right/ \Gamma
$
of $M$, where
$
 H_\cD \coloneqq \bigcup_{\delta \in \Delta(Q)} \delta^\bot
$
is the union of hyperplanes
$
 \delta^\bot \coloneqq \{ [\Omega] \in \cD \mid (\Omega, \delta) = 0 \}.
$
%The closure of the period domain
%in the {\em compact dual}
%$$
% \cDv = \{ [\Omega] \in \bP(Q \otimes \bC) \mid
%  (\Omega, \Omega) = 0 \}
%$$
%of the period domain
%is denoted by $\overline{\cD}$.
%Its topological boundary is given by
%\begin{align}
% \overline{\cD} \setminus \cD
%  = \bigcup_{I \text{ : isotropic subspace of $P_\bR$}}
%      \bP(I_{\bC}) \cap \cD^*.
%\end{align}
%Since the signature of $P$ is $(2, 19-t)$,
%one either has $\rank I = 1$ or 2,
%so that $\bP(I_{\bC}) \cap D_P^*$ is
%one point or isomorphic to the upper half plane.
%The boundary component is \emph{rational}
%if $I$ is defined over $\bQ$.
%The {\em Satake-Baily-Borel compactification}
%is defined by
%$
% M^*
%  = \cD^* / \Gamma,
%$
%where
%$
% \cD^* = \cD \sqcup
%  \bigcup_{I \text{ : rational}}
%      \bP(I_{\bC}) \cap \overline{\cD}.
%$

Let $\cDtilde$ be the connected component of
$
 \{ \Omega \in Q \otimes \bC \mid (\Omega, \Omega) = 0, \ (\Omega, \Omegabar) > 0 \}
$
%which is the total space of a principal $\bCx$-bundle over $\cD$.
projecting onto $\cD$.
A \emph{modular form} of weight $k \in \bZ$
and character $\chi \in \Char(\Gamma) \coloneqq \Hom(\Gamma, \bCx)$
is a holomorphic function
$
 f \colon \cDtilde \to \bC
$
satisfying
\begin{enumerate}[(i)]
 \item
$f(\alpha z) = \alpha^{-k} f(z)$
for any $\alpha \in \bCx$, and
 \item
$f(\gamma z) = \chi(\gamma) f(z)$
for any $\gamma \in \Gamma$.
\end{enumerate}
%\begin{align} \label{eq:mod1}
% f(\lambda z) = \lambda^{-k} f(z)
%\end{align}
%for any $\lambda \in \bCx$ and
%\begin{align} \label{eq:mod2}
% f(\gamma z) = \chi(\gamma) f(z)
%\end{align}
%for any $\gamma \in \Gamma$.
The vector spaces
$A_k(\Gamma, \chi)$
of such modular forms
constitute the ring
\begin{align}
 A(\Gamma)
  \coloneqq \bigoplus_{k=0}^\infty \bigoplus_{\chi \in \Char(\Gamma)}
      A_k(\Gamma, \chi)
\end{align}
of modular forms.
We have $\Char(\Gamma) \cong \mathbb{Z}/2\mathbb{Z}$
 (see Lemmas \ref{lem:gamma1} and \ref{lem:gamma2}).

\section{Moduli space of $T_{2,3,7}$-polarized K3 surfaces}
 \label{sc:shiga}

Let $\bfP = \bfP(6,14,21,1)$ be the weighted projective space
of weight $(6,14,21,1)$ and consider the family
\begin{align}
 \varphi_{\Tbar^*} \colon
 \fYbar
  &= \lc ([x:y:z:w],t) \in \bfP \times \Tbar^* \relmid
   f(x,y,z,w;t) = 0 \rc
  \to \Tbar^* = \bA^{11} \setminus \bszero \label{eq:1}
\end{align}
of hypersurfaces of $\bfP$,
where
\begin{align}
 f(x,y,z,w;t) &= z^2 + y^3 + g_2(x,w;t) y + g_3(x,w;t),
  \label{eq:2} \\
%\end{align}
%Here $g_2(x,w;t)$ and $g_3(x,w;t)$ are defined by
%\begin{align}
\begin{split}
 g_2(x,w;t) &= t_4 x^4 w^4 + t_{10} x^3 w^{10}
  + t_{16} x^2 w^{16} + t_{22} x w^{22} + t_{28} w^{28} ,\\
 g_3(x,w;t) &= x^7 + t_{12} x^5 w^{12} + t_{18} x^4 w^{18}
  + t_{24} x^3 w^{24} + t_{30} x^2 w^{30}
  + t_{36} x w^{36} + t_{42} w^{42},
\end{split} \label{eq:3} \\
%\end{align}
%and
%\begin{align}
 t &= (t_4, t_{10}, t_{12}, t_{16}, t_{18}, t_{22}, t_{24},
  t_{28}, t_{30}, t_{36}, t_{42})
   \in \Tbar^*.
\end{align}
%are parameters.
The group $\bCx$ acts on $\bfP$ and $\Tbar^*$
in such a way that $\alpha \in \bCx$ sends
$[x:y:z:w] \in \bfP$ to $[x:y:z:\alpha^{-1}w]$ and
$t=(t_i)_{i=4}^{42}$ to $\alpha \cdot t = (\alpha^i t_i)_{i=4}^{42}$.
%
%The weight of this $\bGm$-action
%is summarized in \pref{tb:weights}.
%
%\begin{table}[h]
%\begin{align*}
%\begin{array}{ccccc|cccccc}
% s_0 & s_1 & s_2 & s_3 & s_4 &
% t_0 & t_1 & t_2 & t_3 & t_4 & t_5 \\
% \hline
% 28 & 22 & 16 & 10 & 4 &
% 42 & 36 & 30 & 24 & 18 & 12
%\end{array}
%\end{align*}
%\caption{Weights of the $\bGm$-action on $\Tbar = \bA^{11}$}
%\label{tb:weights}
%\end{table}
%
Since $f$ is invariant under the $\bCx$-action,
the family $\varphi_{\Tbar^*} \colon \fYbar \to \Tbar^*$
descends to a family $\varphi_{T^*} \colon \frakY \to T^*$
over the weighted projective space $T^*=\bfP(\bw)$
with weight
\begin{align} \label{eq:weight}
 \bw = (4,10,12,16,18,22,24,28,30,36,42).
\end{align}
The fiber of $\varphi_{\Tbar^*} \colon \fYbar \to \Tbar^*$ over $t \in \Tbar^*$
will be denoted by $\Ybar_t$.
Let $\Tbar$ be the set of $t \in \Tbar^*$
such that $\Ybar_t$ has at worst rational double points,
and $T \coloneqq \Tbar / \bCx$ be the quotient variety.

The following fact is well-known:

\begin{proposition}[{cf.~e.g.~\cite[Proposition I\!I\!I.3.2]{MR1078016}}]
 \label{pr:sing1}
An elliptic surface of the form
\begin{align}
 z^2 + y^3 + g_2(x)y + g_3(x) = 0
\end{align}
has a singularity worse than
rational double points
in the fiber over $x=a$
if and only if $\ord_a(g_2) \ge 4$ and
$\ord_a(g_3) \ge 6$.
\end{proposition}

It follows that $\Ybar_t$ has a singularity
worse than rational double points
if and only if one can set
\begin{align}
\begin{split}
 g_2(x,w) &= a x^4 w^4, \\
 g_3(x,w) &= x^7 + 7 b x^6  w^6
\end{split}
 \label{eq:non_rdp}
\end{align}
by a change of coordinates.
By a simple change of coordinate
from \eqref{eq:non_rdp} to \eqref{eq:3},
%Note that one can normalize $c$ to 1
%by rescaling $w$, and
%remove the $x^6 w^6$-term
%by substituting $(x-b w^6/7)$ to $x$.
%Note that the surface defined by $(a, b)$
%and $(\alpha^4 a, \alpha^6 b)$ for $\alpha \in \bCx$
%are related by rescaling $x$.
%%\begin{align}
%% g_2(x-b w^6,w)
%%  &= a b^4 w^{28} - 4 a b^3 x w^{22} + 6 a b^2 x^2 w^{16}
%%   - 4 a b x^3 w^{10} + a x^4 w^4, \\
%% g_3(x-b w^6,w)
%%  &= 6 b^7 w^{42} - 35 b^6 x w^{36} + 84 b^5 x^2 w^{30}
%%   - 105 b^4 x^3 w^{24} + 70 b^3 x^4 w^{18} - 21 b^2 x^5 w^{12}
%%   + x^7.
%%\end{align}
%As a result,
one obtains the following:

\begin{corollary}
The complement $T^* \setminus T$ consists of points
parametrized as
\begin{align}
 t = [a:-4 a b:- 21 b^2:6 a b^2:70 b^3:-4 a b^3:-105 b^4
  :a b^4:84 b^5:-35 b^6:6 b^7]
\end{align}
%\begin{align}
% g_2(x-b,1)
%  &= a b^4 - 4 a b^3 x + 6 a b^2 x^2 - 4 a b x^3 + a x^4, \\
% g_3(x-b,1)
%  &= 6 b^7 - 35 b^6 x + 84 b^5 x^2 - 105 b^4 x^3
%   + 70 b^3 x^4 - 21 b^2 x^5 + x^7
%\end{align}
for $[a:b] \in \bfP(4,6)$.
\end{corollary}

%Since $\Ybar_t$ is an anti-canonical divisor
%of $\bfP(1,6,14,21)$,
The adjunction formula shows that
$\Ybar_t$ has the trivial canonical sheaf.
Since the minimal resolution of a surface is crepant
%$\Ybar_t$ is a crepant resolution
if and only if it has at worst rational double points,
one obtains the following:

\begin{corollary}
The minimal model $Y_t$ of $\Ybar_t$
is a K3 surface
if and only if $[t] \in T$.
\end{corollary}

Let $P = T_{2,3,7} = E_8 \bot U$ be the lattice
appearing in Introduction.
%whose Coxeter--Dynkin diagram is given in \pref{fg:237-diagram}.

\begin{proposition} \label{pr:polarization}
The minimal model $Y_t$ for any $[t] \in T$
has a natural structure of a pseudo-ample
$P$-polarized K3 surface.
\end{proposition}

\begin{proof}
The divisor
$
 \Ybar_t \cap \{ w = 0 \}
$
of $\Ybar_t$
at infinity is given by
\begin{align}
 Y_\infty = \lc [x:y:z] \in \bfP(6,14,21) \relmid x^7+y^3+z^2=0 \rc,
\end{align}
which is a rational curve.
The hypersurface $\Ybar_t$ has
\begin{itemize}
 \item
an $A_6$-singularity at $[0:-1:1:0]$,
 \item
an $A_2$-singularity at $[-1:0:1:0]$, and
 \item
an $A_1$-singularity at $[1:-1:0:0]$,
\end{itemize}
all coming from the singularity
of the ambient space $\bfP$.
It follows that
the minimal resolution $Y_t \to \Ybar_t$
has a configuration of $(-2)$-curves,
whose dual intersection graph is given
by $T_{2,3,7}$.
% shown in \pref{fg:237-diagram}.
\end{proof}

We call the singularities appearing in the proof
of \pref{pr:polarization}
as \emph{generic singularities of the family}.

\begin{proposition}
The $P$-polarized K3 surface $Y_t$ for $[t] \in T$ is ample
if and only if $\Ybar_t$ has only generic singularities of the family.
\end{proposition}

\begin{proof}
A $P$-polarized K3 surface $Y_t$ is strictly pseudo-ample
if and only if its period $\Omega = H^{2,0}(Y_t)$ is on the reflection hyperplane
$H_\delta$ of a root $\delta \in \Delta(Q)$.
This happens if and only if the element $\delta$ or $-\delta$
considered as a cohomology class
by the embedding $Q \subset H^2(Y_t;\bZ)$
 (which is induced by the $P$-polarization)
is Poincar\'{e} dual to a $(-2)$-curve,
since the orthogonal lattice of $\Omega$
inside $H^2(Y_t;\bZ)$ is the N\'{e}ron-Severi lattice.
This $(-2)$-curve is contracted in $\Ybar_t$
since $\delta$ is orthogonal to $P$,
so that it appears as a singularity of $\Ybar_t$,
which must be distinct from the generic singularities in the family
since the class $\delta$ is not contained in $P$.
\end{proof}

\begin{proposition} \label{pr:surjectivity}
For any pseudo-ample $P$-polarized K3 surface $Y$,
there exist $[t] \in T$ and an isomorphism
$Y \simto Y_t$
of pseudo-ample $P$-polarized K3 surfaces.
%In other words,
%the period map
%$
% \Pi : T \to  M
%$
%is surjective.
\end{proposition}

\begin{proof}
We identify $P$ with its image
by $j \colon P \hookrightarrow \Pic(Y)$.
Choose a basis $\{ e, f \}$
of the orthogonal summand $U$ of
$P = U \bot E_8$
in such a way that
$
 (e, e) = (f, f) = (e,f)-1 = 0
$
and $f \in C(P)$.
The pseudo-ampleness of $Y$ implies that
$f$ is nef.
Then one can show
(cf.~\cite[\S 3, Theorem 1]{MR0284440})
that $Y$ admits a unique structure of an elliptic K3 surface
with a section
such that $f$ is the class of a fiber and
$e-f$ is the class of a section.

An elliptic K3 surface with a section
admits a Weierstrass model of the form
\begin{align} \label{eq:weierstrass1}
 z^2 + y^3 + g_2(x,w) y + g_3(x,w) = 0
\end{align}
in $\bP(1,4,6,1)$
(cf.~e.g.~\cite[Section 4]{MR2732092}).
Since the sublattice $E_8 \subset P$ is orthogonal to $f \in U$,
it is generated by irreducible components
of a fiber of Kodaira type I\!I*.
One can choose a coordinate
in such a way that this fiber lies over
the point $x = \infty$ (or $w = 0$) in $\bP^1$.
In order for the elliptic surface \eqref{eq:weierstrass1}
to have a singular fiber of type I\!I* at $\infty$,
one needs
\begin{align}
 \ord_\infty g_2(x,w) \ge 4, \quad
 \ord_\infty g_3(x,w) = 5, \ \text{ and } \ 
 \ord_\infty \Delta(x,w) = 10,
\end{align}
where
$
 \Delta = 4 g_2^3 + 27 g_3^2
$
(cf.~e.g.~\cite[Table IV.3.1]{MR1078016}).
This requires
\begin{align}
 g_2(x,w) &= \sum_{i=0}^4 s_i x^i w^{8-i}, \\
 g_3(x,w) &= \sum_{i=0}^7 s'_j x^j w^{12-j}.
\end{align}
%If $s'_7=0$,
%then $Y$ has a singularity worse than rational double point by \pref{pr:sing1}.
Since $\ord_\infty g_3(x,w) = 5$, one can set $s'_7 = 1$ and $s'_6 = 0$
by a change of coordinates of $(x,w)$.
The birational map
\begin{align}
 \lb \frac{x}{w}, \frac{y}{w^4}, \frac{z}{w^6} \rb
  \mapsto
 \lb \frac{x}{w^6}, \frac{y}{w^{14}}, \frac{z}{w^{21}} \rb
\end{align}
of the ambient space
from $\bfP(1,4,6,1)$ to $\bfP(6,14,21,1)$
sends \eqref{eq:weierstrass1}
to \eqref{eq:1}.
This map is compatible with $P$-polarizations
because the map $j \colon P \hookrightarrow \Pic(Y)$
is uniquely determined by identifying the configuration of
$(-2)$-curves contained in the fiber over $x=\infty$ % $Y\setminus \{ w\neq 0 \}$
with that in \pref{fg:237-diagram}.
Hence
\pref{pr:surjectivity} is proved.
\end{proof}

The proof of \pref{pr:surjectivity} also shows the following:

\begin{proposition} \label{pr:auto1}
For an isomorphism $\phi \colon Y_t \to Y_{t'}$
of pseudo-ample $P$-polarized K3 surfaces,
there exists an element $\alpha \in \bCx$
such that the following diagram commutes;
\begin{align}
\begin{CD}
 Y_t @>{\varphi}>> \Ybar_t @>{\iota}>> \bfP(6,14,21,1) \\
 @V{\phi}VV @. @VV{\phi_\alpha}V \\
 Y_{t'} @>{\varphi'}>> \Ybar_{t'} @>{\iota'}>> \bfP(6,14,21,1).
\end{CD}
\end{align}
Here $\varphi$ and $\varphi'$ are minimal resolutions,
$\iota$ and $\iota'$ are inclusions,
and $\phi_\alpha$ is the automorphism of $\bfP(6,14,21,1)$
sending $[x:y:z:w]$ to $[x:y:z:\alpha w]$.
\end{proposition}

\begin{proof}
The proof of \pref{pr:surjectivity} shows that
the embeddings $\iota$ and $\iota'$
given by the Weierstrass models \eqref{eq:1}
are determined by the pseudo-ample $P$-polarization
up to an automorphism of $\bfP(1,6,14,21)$.
The automorphism group of $\bfP(1,6,14,21)$ consists of
transformations of the form
\begin{align}
 x &\mapsto \beta_1 x + \beta_2 w^6, \\
 y &\mapsto \gamma_1 y + \gamma_2 x w^8 + \gamma_3 x^2 w^2, \\
 z &\mapsto \delta_1 z + \delta_2 x y w + \delta_3 x w^{15}
  + \delta_4 x^2 w^{9} + \delta_5 x^3 w^3, \\
 w &\mapsto \alpha w.
\end{align}
The only automorphism which preserves
the Weierstrass model \eqref{eq:1}
is $w \mapsto \alpha w$.
\end{proof}

By applying the global Torelli theorem
\cite{MR0284440, MR0447635}
and the surjectivity of the period map
\cite{MR592693},
one concludes that
the period map
$
 \Pi_T \colon T \to M
$, which is induced by $[t] \mapsto H^{2,0}(Y_t)$,
is an isomorphism.
We study the period map $\Pi_T$ and its lifts in details in the next section.

%\begin{corollary}%[{Looijenga \cite{MR761312}}]
% \label{cr:Looijenga}
%The period map
%$
% \Pi \colon T \to M
%$
%is an isomorphism.
%\end{corollary}

%Note that there is a ``generic'' monodromy for our family.
%Namely, when $\alpha=1$ goes to $\alpha=-1$,
% we have the same K3 surface for $(\alpha^i t_i)$, $\alpha\in \{\pm 1\}$
% because the $t_i$ are all even,
% but these K3 surfaces should be identified by $w\mapsto -w$
% (or, equivalently, $z \mapsto -z$).

%%\pref{pr:subjectivity} and \pref{pr:auto1} implies the following:
%The discussions so far is summarized as follows:
%
%\begin{theorem}[{\cite{Shiga_K3MIII}}]
%The period map gives an isomorphism
%$
% \Pi \colon T \to M.
%$
%\end{theorem}
%
%%\pref{pr:auto1} shows that
%%the period map
%%$
%% \Pi : T \to  M
%%$
%%is injective,
%%and \pref{th:shiga}.2 is proved.
%%This concludes the proof of \pref{th:shiga}.2.

\section{Orbifold structure in codimension one}
 \label{sc:H}

%The biholomorphic extension of the period map
%in \pref{th:shiga}.3 follows as a corollary.

%Recall that $\cDtilde$ is the connected component of
%$
% \lc \Omega \in Q \otimes \bC \relmid (\Omega, \Omega) = 0, \ 
%  (\Omega, \Omegabar) > 0 \rc
%$
%which projects to $\cD$.
The natural projection
$
 \pi \colon \cDtilde \to \cD
$
is a principal $\bCx$-bundle,
which is trivial
since it admits a section
\begin{align} \label{eq:section}
 \lc e - ((v,v)/2) f + v \in \cDtilde \relmid
  v = v_1 + \sqrt{-1} v_2 \in P \otimes \bC, \ 
  (v_2, v_2) > 0 \rc.
\end{align}
It induces a principal $\bCx$-orbi-bundle
\begin{align} \label{eq:[pi]}
 [\pi] \colon \ld \left. \cDtilde \right/ \Gamma \rd \to \bM
  \coloneqq \ld \left. \cD \right/ \Gamma \rd
  \cong \ld \left. \cDtilde \right/ \lb \Gamma \times \bCx \rb \rd
\end{align}
since $\pi$ is equivariant
with respect to the natural action of
$\Gamma$.
The section \eqref{eq:section} of $\pi$
is not a section of $[\pi]$
since it is not $\Gamma$-equivariant.
The line bundle associated with
the principal $\bCx$-orbi-bundle \pref{eq:[pi]}
%$[\pi] \colon \ld \left. \cDtilde \right/ \Gamma \rd \to \bM$
will be denoted by $\cO_\bM(1)$.

The fixed locus $\cD^g$
of an element $g \in \Gammabar$
is the intersection of a proper linear subspace and $\cD$.
The element $g$ is said to be a \emph{reflection}
if $\cD^g$ is the intersection of a hyperplane with $\cD$.

\begin{lemma} \label{lem:gamma1}
Any reflection in $\Gammabar$ is given by
$
 z \mapsto z + (z, \delta) \cdot \delta
$
for some $\delta \in \Delta(Q)$.
\end{lemma}

\begin{proof}
We may assume $\rank (Q^g)^\bot=1$
by taking $-g$ instead of $g$ if necessary.
Let $\delta\neq 0$ be a primitive element in $(Q^g)^\bot$.
Then we have $g(\delta)=-\delta$ and $g$ is given by
 $z \mapsto z - (2(z, \delta)/(\delta,\delta)) \cdot \delta$.
Since $Q$ is unimodular, there exists an element $z\in Q$
 such that $(z,\delta)=1$.
Since $g(z)\in Q$, it follows that $2/(\delta,\delta)\in \bZ$.
Hence $(\delta,\delta)=\pm 2$.
Since $g\in O^+(Q)$, we have $(\delta,\delta)=-2$.
\end{proof}

%\begin{lemma}
%$\Gamma$ is generated by reflections
%along elements of $\Delta(Q)$.
%\end{lemma}
%
%\begin{proof}
%(To be filled in.)
%\end{proof}

\begin{lemma} \label{lem:gamma2}
The action of $\Gamma$ on $\Delta(Q)$ is transitive.
\end{lemma}

\begin{proof}
For any $\delta_1,\delta_2\in \Delta(Q)$,
it follows from Nikulin's theory of discriminant forms of lattices 
\cite{MR525944}
that
$
 \delta_i^\bot \cong \langle 2 \rangle \bot U \bot E_8
$
and that there exists an element $g\in O(Q)$ such that
$g(\delta_1)=\delta_2$.
We may assume $g\in O^+(Q)$
since $O^+(\delta_i^\bot) \subsetneq O(\delta_i^\bot)$.
\end{proof}

Now we construct the following diagram:
\begin{align} \label{eq:diagram}
\begin{CD}
 \Utilde @>>> \Ubar @>>> U @>>> V @>>> T \setminus S_T \\
 @VV{\Pi_\Utilde}V @VV{\Pi_\Ubar}V @VV{\Pi_U}V @VV{\Pi_V}V @VV{\Pi_T}V \\
 \cDtilde \setminus S_{\cDtilde}
  @>{\{\pm 1\}}>> \cDbar \setminus S_{\cDbar}
  @>{\bCx/\{ \pm 1 \}}>> \cD \setminus S_\cD
  @>{\Gammabar_2}>> M_2 \setminus S_{M_2}
  @>{\bsmu}>> M \setminus S_M
\end{CD}
\end{align}
where horizontal arrows are principal bundles and
vertical arrows are isomorphisms.

%Recall that
%$
% H_{\cD}
%  \coloneqq \cup_{\delta \in \Delta(Q)} \delta^\bot
%  \subset \cD
%$
%is the union of reflection hyperplanes.
%$
% \delta^\bot \coloneqq \lc z \in Q \otimes \bC \relmid (z,\delta) = 0 \rc,
%$
Let
$
 S_\cD \subset \cD
$
%$S_{\cDtilde} \subset \cDtilde$
be the locus where the stabilizer of the action of
$
 \Gammabar
$
is non-trivial
and does not coincide with a group of order two
generated by a reflection.
The locus $S_\cD$
contains
not only intersections
%$
% S_\cD = \bigcup_{\delta \ne \delta' \in \Delta(Q)}
%  H_\delta \cap H_{\delta'} \cap \cD^+
%$
of more than two reflection hyperplanes,
but also points where the corresponding lattice polarized K3 surface
has an automorphism other than
$
 [x:y:z:w] \mapsto [x:y:z:-w] = [x:y:-z:w].
$
%If we write the pull-back of $S_\cD$ to $\cDtilde$ as $S_{\cDtilde}$,
%then
The action of $\Gammabar$
on $\cD \setminus (H_{\cD} \cup S_{\cD})$ is free,
and the stabilizer of a point in
$H_{\cD} \setminus (H_{\cD} \cap S_{\cD})$
is a group of order two
generated by a reflection.

Since the group $\Gammabar$ is countable,
%one has the following:
%
%\begin{lemma}
the locus $H_{\cD}$ is a countable union of hyperplanes, and
the locus $S_{\cD}$ is a countable union of linear subspaces
of codimension greater than one.
%\end{lemma}
%
The images of $H_\cD$ and $S_\cD$ in $M$
will be denoted by $H_M$ and $S_M$ respectively.

Let $\Gammabar_2$ and $\bsmu$ be the kernel and the image of the determinant map
$
 \det \colon \Gammabar \to \bCx.
$
%Although $\bsmu$ and $\bsmu'$ are isomorphic as groups,
%it is convenient to distinguish them.
%to avoid confusion.
Set $M_2 \coloneqq \cD / \Gammabar_2 $,
$H_{M_2} \coloneqq H_\cD/\Gammabar_2$, and
$S_{M_2} \coloneqq S_\cD/\Gammabar_2$.
The map
$
 M_2 \setminus S_{M_2}
  \to M \setminus S_M
$
is a double cover
branched along $H_M \setminus S_M$,
and the map
$
 \cD \setminus S_\cD
  \to M_2 \setminus S_{M_2}
$
is the universal cover
since the action of $\Gammabar_2$
on $\cD \setminus S_\cD$ is free.

Let $H_T$ and $S_T$
be the inverse images of $H_M$ and $S_M$
by the period map
$
 \Pi \colon T \simto M.
$
The pull-backs of $H_T$ and $S_T$
by the projection $\Tbar \to T$
will be denoted by $H_\Tbar$ and $S_\Tbar$.
Let further $\Delta_T \in \bC[t]$ be the defining equation of $H_T$, and
\begin{align} \label{eq:V}
 V = \Vbar / \bCx, \text{~where~}
 \Vbar = \lc (t,s) \in (\Tbar \setminus S_\Tbar) \times \bA^1
  \relmid s^2 = \Delta_T(t) \rc,
\end{align}
be the double cover of $T \setminus S_T$
branched along $H_T \setminus S_T$.
Here $\alpha \in \bCx$
acts on $\bA^1$
in \pref{eq:V}
by
$
 s \mapsto \alpha^{\deg \Delta_T/2} s.
$
(We will show $\deg \Delta_T=504$ in Proposition \ref{pr:degree}.)
The restriction
$
 \Pi|_{T \setminus S_T} \colon T \setminus S_T \simto M \setminus S_M
$
of the period map
lifts to an isomorphism
$
 \Pi_V \colon V \to M_2 \setminus S_{M_2},
$
since both $V$ and $M_2 \setminus S_{M_2}$
are the double covers of isomorphic
% smooth
 varieties branched along smooth divisors
which are identified under the isomorphism.
%The passage
%from $T \setminus S_T$
%to the orbifold quotient $[V/\bsmu]$
%is the \emph{root construction}
%\cite{MR2450211,Cadman_US}
%of order 2
%along $H_T \setminus S_T$.

The isomorphism $\Pi_V$ can further be lifted
to an isomorphism
$
 \Pi_U \colon U \to \cD \setminus S_\cD,
$
where $U$ is the universal cover of $V$.
This isomorphism is equivariant
under the action of the covering transformation group
$
 \Gal \lb (\cD \setminus S_\cD)
   \relmiddle/ (M \setminus S_M) \rb
  = \Gammabar.
$

Set $(x_1,x_2,x_3,x_4)=(x,y,z,w)$
and $(q_1,q_2,q_3,q_4)=(6,14,21,1)$.
The Griffiths--Dwork method provides the element
\[
 \Omegabar
  = \Res \frac{\sum_{i=1}^4 (-1)^i q_i x_i
  d x_1 \wedge \cdots \wedge \widehat{d x_i}
   \wedge \cdots \wedge d x_4}{f}
\]
of $H^0 \lb \Omega_{\fYbar/\Tbar}^2 \rb$,
which gives a 2-form $\Omegabar_t$ on $\Ybar_t$ for each $t \in \Tbar$,
whose pull-back to $Y_t$ gives a holomorphic 2-form $\Omega_t$.

Let $\fL_{\Tbar} \to \Tbar \setminus (H_\Tbar \cup S_\Tbar)$
be the local system
whose fiber over $t \in \Tbar \setminus (H_\Tbar \cup S_\Tbar)$
is the second homology group $H_2(Y_t; \bZ)$.
For any $t \in \Tbar \setminus (H_\Tbar \cup S_\Tbar)$,
the fiber of the projection $\Tbar \to T$
above $[t] \in T$ can be identified
with $\bCx/\{ \pm 1 \}$, and
the monodromy of $\fL_{\Tbar}$
along the generator of $\pi_1(\bCx/\{ \pm 1 \}) \cong \bZ$
is %$- \id_{H_2(Y_t; \bZ)}$
induced by the automorphism
$[x:y:z:w] \mapsto [x:y:z:-w]$
of $\Ybar_t$,
 whose induced action on $Q$ is $- \id_Q$.

The monodromy of $\fL_\Tbar$ along $H_\Tbar$
is the Picard--Lefschetz transformation
with respect to the vanishing cycle $C$,
which is the reflection
along the homology class $[C]$ of the vanishing cycle.
Note that the class $[C]$ is equal
to the class of the $(-2)$-curve
which defines the reflection hyperplane.
It follows that
the pull-back of $\fL_\Tbar$
to the double cover $\Vbar$ %\coloneqq V \times_T \Tbar$
does not have a monodromy
along the ramification divisor $H_{\Vbar}$,
and hence extends to a local system $\fL_{\Vbar}$
on $\Vbar$.

Let $\fL_\Ubar$ be the pull-back
of $\fL_\Vbar$ to $\Ubar \coloneqq U \times_T \Tbar$.
One has $\pi_1(\Ubar) \cong \pi_1(\bCx/\{ \pm 1\})$
since $U$ is the universal cover of $V$.
The monodromy of $\fL_\Ubar$
along the generator of $\pi_1(\Ubar) \cong \pi_1(\bCx/\{ \pm 1\})$
is given by $- \id_{H_2(Y_t; \bZ)}$.
It follows that the pull-back $\fL_\Utilde$ of $\fL_\Ubar$
to the non-trivial double cover $\Utilde \to \Ubar$
is trivial.
The period map $\Pi_T \colon T \to M$ is defined
in such a way that the lift
$\Pi_\Utilde \colon \Utilde \to \cDtilde \setminus S_{\cDtilde}$
is given by integration of $\Omega_t$
along a basis of $\fL_\Utilde$
obtained by choosing a global trivialization of $\fL_\Utilde$.
It follows from the construction that
the map $\Pi_{\Utilde}$ is equivariant
with respect to the natural action of the central extension
$\Gamma$ of $\Gammabar$ by $\{ \pm 1 \}$,
where $\Gammabar$ acts on the bases
of the principal $\bCx$-bundles
$
 \lb \Utilde \to U \rb
  \cong \lb \cDtilde \setminus S_{\cDtilde} \to \cD \setminus S_\cD \rb,
$
and $\{ \pm 1 \}$ acts on the fibers.
The map $\Pi_{\Utilde}$ induces an isomorphism
$
 \Pi_{\Ubar} \colon \Ubar \to \cDbar \setminus S_{\cDbar},
$
where $\cDbar \coloneqq \cDtilde / \{ \pm 1\}$
and $S_{\cDbar} \coloneqq S_{\cDtilde} / \{ \pm 1 \}$.

The action of $\alpha \in \bCx$ on $\Tbar$
sends a point $(t_i)_i$ to $(\alpha^i t_i)_i$.
The change of $f$ caused by this action
can be absorbed by the coordinate change
sending $w \mapsto \alpha^{-1} w$
and keeping $x$, $y$ and $z$ fixed.
This sends $\Omega$ to $\alpha^{-1} \Omega$,
so that the period will be multiplied by $\alpha^{-1}$.
This shows that $\Pi_\Utilde$ is $\bCx$-equivariant.
%and \pref{pr:isom} is proved.
%\end{proof}

%The situation is summarized as
%where horizontal arrows are principal bundles and
%vertical arrows are isomorphisms.

Let
%$H^*$ be the closure of $\Pi_T^{-1}(H_M)$ in $T^* =\bfP(\bsw)$, and
$H_{\bP}$ be the hypersurface
of $\bP(\bsw) \coloneqq \ld \Tbar / \bCx \rd$
defined by $\Delta_T$.
The orbifold $\bP(\bsw)$ has a generic stabilizer of order 2,
and the natural morphism
to the orbifold
$
% \bP \to
  \bP(\bsw/2) \coloneqq \ld \Tbar / (\bCx / \{ \pm 1 \}) \rd
$
without a generic stabilizer
is a $B \{ \pm 1\}$-bundle.
Let $\bT^*$ be the stack
obtained from $\bP(\bsw)$ by the \emph{root construction}
\cite{MR2450211,Cadman_US}
of order 2
along the divisor $H_{\bP}$.
The orbifold quotient
\begin{align}
 \bT
  \coloneqq \ld \left. \Vbar \right/ \lb \bsmu \times \bCx \rb \rd
  \cong \ld \left. \Utilde \right/ \lb \Gamma \times \bCx \rb \rd
\end{align}
%where
%$
% \Vbar \coloneqq \Tbar \times_T V,
%$
is an open substack of $\bT^*$.
Since $\Pi_\Utilde$ is $\Gamma \times \bCx$-equivariant,
one obtains an isomorphism
\begin{align} \label{eq:Pi_bT}
\begin{split}
 \Pi_{\bT}
  \colon \bT
%  &\coloneqq \ld \left. \lb \Ttilde \setminus S_\Ttilde \rb \right/ \bCx \rd
%  \cong \ld \left. \Utilde \right/ \lb \Gamma \times \bCx \rb \rd \\
  & \simto \bM \setminus S_\bM
%   \coloneqq \ld \left. \lb \cDtilde \setminus S_{\cDtilde} \rb \right/ \lb \Gamma \times \bCx \rb \rd
\end{split}
\end{align}
of orbifolds
where
$
 S_\bM \coloneqq \ld S_{\cDtilde} / (\Gamma \times \bCx) \rd.
$
%satisfying
%$
% [\Pi]^* \cO_{\bM \setminus S_\bM}(1)
%  \cong \cO_{\bT \setminus S_\bT}(1).
%$
Since the codimension of $S_{\cDtilde}$ in $\cDtilde$ is greater than one,
\pref{th:main} is proved.

\section{The canonical bundle and the total coordinate ring}
 \label{sc:cor}

A character of $\bCx \times \Gamma$ gives
a $\bCx \times \Gamma$-equivariant structure
on the trivial line bundle on $\cDtilde$,
which in turn gives a line bundle
on the orbifold quotient
$
 \bM \cong \ld \left. \cDtilde \right/ \lb \bCx \times \Gamma \rb \rd.
$
We write the line bundle on $\bM$
associated with the character
$
 \bCx \times \Gamma \ni (\alpha, g) \mapsto \alpha^{-k} \cdot (\det g)^l
$
as
$
 \cO_\bM(k) \otimes \det^{l}.
$

\begin{proposition} \label{pr:canonical}
The canonical bundle on $\bM$ is given by
$
 \omega_\bM \cong \cO_\bM(10) \otimes \det.
$
\end{proposition}

\begin{proof}
We first consider the canonical bundle of
$
 \cD \cong \ld \left. \cDtilde \right/ \bCx \rd,
$
which is an open subset defined by $\lb \Omega, \Omegabar \rb > 0$
of a quadratic hypersurface in $\bP(Q \otimes \bC)$
defined by $(\Omega, \Omega) = 0$.
In general, the canonical bundle
of a degree $k$ hypersurface $X$ in $\bP^n$
is given by $\cO_X(k-n-1)$.
We follow our convention for $\cO_{\bM}(k)$
and write the line bundle on
$
 \cD \cong \ld \left. \cDtilde \right/ \bCx \rd
$
associated with the character $\alpha \mapsto \alpha^{-k}$ of $\bCx$
as $\cO_{\cD}(k)$.
Since this is inverse to the usual convention,
one has $\omega_\cD \cong \cO_\cD(10)$.

Now we take the action of $\Gamma$ into account
and consider the canonical bundle of $\bM \cong \ld \left. \cDtilde \right/ (\bCx \times \Gamma) \rd$.
Since a reflection changes the sign of a top differential form,
one concludes that $\omega_\bM \cong \cO_\bM(10) \otimes \det$.
\end{proof}

It follows from \pref{th:main}
that $\Pic \bM$ is isomorphic to $\Pic \bT^*$.
%Recall that $\bT^*$ is obtained from $\bP(\bsw)$
%by the root construction along the divisor $H_{\bP} \subset \bP(\bsw)$
%defined by the defining equation $\Delta_T$ of $H_T$.
Let
$
 \cO_{\bT^*}(1)
  \coloneqq p^* \cO_{\bP}(1)
$
be the pull-back
of the positive generator $\cO_{\bP}(1)$
of $\Pic \bP(\bsw) \cong \bZ$
by the structure morphism
$
 p \colon \bT^* \to \bP(\bsw),
$
and
$
 \cO_{\bT^*}(H_{\bT^*})
$
be the tautological bundle
on the root stack
satisfying
\begin{align} \label{eq:root_rel}
 \cO_{\bT^*}(H_{\bT^*})^{\otimes 2}
  \cong p^* \cO_{\bP}(H_\bP).
%  \big( \cong \cO_{\bT^*}(\deg \Delta_T) \big).
\end{align}
It follows from \cite[Section 3.1]{Cadman_US}
that $\Pic \bT^*$ is generated by
$
 \cO_{\bT^*}(1)
%  \coloneqq p^* \cO_{\bP}(1)
$
%of the positive generator $\cO_{\bP}(1)$
%of $\Pic \bP(\bsw) \cong \bZ$
%by the structure morphism
%$
% p \colon \bT^* \to \bP
%$
and
%the tautological bundle
$
 \cO_{\bT^*}(H_{\bT^*})
$
%on the root stack
with relation \pref{eq:root_rel}.
%$
% \cO_{\bT^*}(H_{\bT^*})^{\otimes 2}
%  \cong \cO_{\bT^*}(\deg \Delta_T)
%  \coloneqq \cO_{\bT^*}(1)^{\otimes (\deg \Delta_T)}.
%$
This shows that $\Pic \bT^*$,
and hence $\Pic \bM$,
is isomorphic to $\bZ \oplus \bZ/2\bZ$.
The free part is generated by $\cO_{\bT^*}(1)$,
which is isomorphic to $\Pi_{\bT}^* \cO_{\bM}(1)$
since $\Pi_{\bT}$ comes from the $(\bCx \times \Gamma)$-equivariant morphism
$\Pi_{\Utilde}$.
The identification of the torsion part is determined uniquely
by the group structure of the Picard group as
$
 \Pi_{\bT}^* \lb \cO_\bM \otimes \det \rb
 \cong \cO_{\bT^*}(-\deg \Delta_T/2) \otimes \cO_{\bT^*}(H_{\bT^*}).
$

\begin{proposition} \label{pr:degree}
One has $\deg \Delta_T = 504$.
\end{proposition}

\begin{proof}
Since the root stack is the quotient of the branched double cover,
the ramification formula for the canonical bundle gives
\begin{align*}
 \omega_{\bT^*}
  &\cong p^* \omega_{\bP} \otimes \cO_{\bT^*} (H_{\bT^*}) \\
  &\cong \cO_{\bT^*}(-242) \otimes \cO_{\bT^*}(H_{\bT^*}) \\
  &\cong \cO_{\bT^*}(-242+\deg \Delta_T/2)
    \otimes (\cO_{\bT^*}(-\deg \Delta_T/2) \otimes \cO_{\bT^*}(H_{\bT^*})).
\end{align*}
Since this is isomorphic to
$
 \Pi_{\bT}^* \omega_{\bM},
$
one has
$
 -242+\deg \Delta_T/2=10
$
and hence
$
 \deg \Delta_T=504.
$
\end{proof}

One has
\begin{align}
 A(\Gamma)
  &\coloneqq \bigoplus_{k=0}^\infty \bigoplus_{\chi \in \Char(\Gamma)}
      A_k(\Gamma, \chi) \\
  &\cong \bigoplus_{\cL \in \Pic \bM} H^0(\cL) \\
  &\cong \bigoplus_{\cL \in \Pic \bT^*} H^0(\cL),
\end{align}
which is generated by the polynomial ring
$
 \bigoplus_{k=0}^\infty H^0(\cO_{\bT^*}(k))
  \cong \bigoplus_{k=0}^\infty H^0(\cO_{\bP}(k))
$
and the canonical section
$
 s \in H^0 \lb \cO_{\bT^*}(H_{\bT^*}) \rb
$
satisfying $s^2 = \Delta_T$.
This concludes the proof of \pref{cr:main}.

\section{Discriminant and resultant}
 \label{sc:disc}

The discriminant of
$
 y^3 + g_2(x,w;t) y + g_3(x,w;t)
$
as a polynomial of $y$ is given by
$
 4 g_2(x,w;t)^3 + 27 g_3(x,w;t)^2,
$
which defines a hypersurface
of degree 84 in $\bP(6,1) = \bProj \bC[x,w]$.
In other words,
$
 [4 g_2(x,w;t)^3 + 27 g_3(x,w;t)^2]/w^{84}
$
is a polynomial of degree 14
in $x/w^6$.
Let $k(t)$ be the discriminant of this polynomial in one variable,
which is a homogeneous polynomial
of degree $14 \cdot 13 \cdot 6 = 1092$ in $t$.
A general point
on the divisor $D \subset \bfP(\bw)$
defined by $k(t)$
corresponds to the locus
where two fibers of Kodaira type $\mathrm{I}_1$ collapse
into one fiber.
The divisor $D$ is a linear combination
of two prime divisors $D_1$ and $D_2$.
A general point on the component $D_1$ corresponds to the case
when there exists a point $p = [x:w] $ on $\bP(6,1)$ such that
neither $g_2$ nor $g_3$ vanishes at $p$,
and a general point on the other component $D_2$ corresponds to the case
when both $g_2$ and $g_3$ vanish at $p$.
In the former case,
the resulting singular fiber is of Kodaira type $\mathrm{I}_2$,
and the surface $\Ybar_t$ acquires an $A_1$-singularity.
In the latter case,
the resulting singular fiber is of Kodaira type I\!I,
and the surface $\Ybar_t$ does not acquire any new singularity.
The defining equation of $D_1$ is $\Delta_T$.
The defining equation of $D_2$
is the resultant of $g_2$ and $g_3$,
which is given as the determinant
%\begin{align}
%r=
%\left|
%\begin{array}{cccccccccccc}
% t_4 & t_{10} & t_{16} & t_{22} & t_{28} \\
% & t_4 & t_{10} & t_{16} & t_{22} & t_{28} \\
% && t_4 & t_{10} & t_{16} & t_{22} & t_{28} \\
% &&& t_4 & t_{10} & t_{16} & t_{22} & t_{28} \\
% &&&& t_4 & t_{10} & t_{16} & t_{22} & t_{28} \\
% &&&&& t_4 & t_{10} & t_{16} & t_{22} & t_{28} \\
% &&&&&& t_4 & t_{10} & t_{16} & t_{22} & t_{28} \\
% 1 & 0 & t_{12} & t_{18} & t_{24} & t_{30} & t_{36} & t_{40} \\
% & 1 & 0 & t_{12} & t_{18} & t_{24} & t_{30} & t_{36} & t_{40} \\
% && 1 & 0 & t_{12} & t_{18} & t_{24} & t_{30} & t_{36} & t_{40} \\
% &&& 1 & 0 & t_{12} & t_{18} & t_{24} & t_{30} & t_{36} & t_{40}
%\end{array}
%\right|
%\end{align}
\begin{align}
f_2=
\left|
\begin{array}{cccccccccccc}
 t_{28} & t_{22} & t_{16} & t_{10} & t_4 \\
 & t_{28} & t_{22} & t_{16} & t_{10} & t_4 \\
 && t_{28} & t_{22} & t_{16} & t_{10} & t_4 \\
 &&& t_{28} & t_{22} & t_{16} & t_{10} & t_4 \\
 &&&& t_{28} & t_{22} & t_{16} & t_{10} & t_4 \\
 &&&&& t_{28} & t_{22} & t_{16} & t_{10} & t_4 \\
 &&&&&& t_{28} & t_{22} & t_{16} & t_{10} & t_4 \\
 t_{42} & t_{36} & t_{30} & t_{24} & t_{18} & t_{12} & 0 & 1 \\
 & t_{42} & t_{36} & t_{30} & t_{24} & t_{18} & t_{12} & 0 & 1 \\
 && t_{42} & t_{36} & t_{30} & t_{24} & t_{18} & t_{12} & 0 & 1 \\
 &&& t_{42} & t_{36} & t_{30} & t_{24} & t_{18} & t_{12} & 0 & 1
\end{array}
\right|
\end{align}
of the Sylvester matrix,
which is homogeneous of degree $d_2=196$.

\begin{lemma} \label{lm:order}
For a polynomial
$
 f(x,y,t) = y^3 + g_2(x,t) y + g_3(x,t)
$
in three variables,
%with coefficients in $\bC[t]$,
let
$
 h(x,t) = 4 g_2(x,t)^3 + 27 g_3(x,t)^2
%  \in \bC[x,t]
$
be the discriminant of $f(x,y,t)$
as a polynomial in $y$,
$k(t)$ be the discriminant of $h(x,t)$
as a polynomial in $x$, and
$r(t)$ be the resultant of the pair
$(g_2(x,t), g_3(x,t))$ as polynomials of $x$.
Then one has $k(t) = r(t)^3 \cdot \ell(t)$
for some polynomial $\ell(t) \in \bC[t]$.
\end{lemma}

\begin{proof}
We may assume that coefficients of $g_2(x,t)$ and $g_3(x,t)$ are generic.
In general,
if we set $h(x,t) = g_2(x,t)^n - g_3(x,t)^m$
for $n > m \ge 1$,
then the order of vanishing of the discriminant $k(t)$
of $h(x,t)$ along the resultant $r(t)$ of $g_2(x,t)$ and $g_3(x,t)$
is given by $n(m-1)$.
This follows from the fact that
the set of solutions of
\begin{align}
 (x-\alpha)^n-(x-\beta)^m=0
\end{align}
for $n > m$ near $\alpha = \beta$
consists of $m$ solutions of the form
\begin{align}
 a_i = \beta + \zeta_m^i (\beta-\alpha)^{n/m} + o((\beta-\alpha)^{n/m}),
  \quad i = 0, \ldots, m-1
\end{align}
and $n-m$ solutions of the form
\begin{align}
 b_j = \alpha + \zeta_{n-m}^j - \frac{m}{n-m} (\beta-\alpha) + o(\beta-\alpha),
  \quad j=0,\ldots,n-m-1,
\end{align}
so that the leading term of the discriminant 
\begin{align}
 \lb \prod_{i<i'}(a_i-a_{i'})^2 \rb \cdot
 \lb \prod_{j<j'}(b_j-b_{j'})^2 \rb \cdot
 \lb \prod_{i,j}(a_i-b_j)^2 \rb
\end{align}
is given by
\begin{align}
 \prod_{i<i'}(a_i-a_{i'})^2
  \sim ((\beta-\alpha)^{n/m})^{m(m-1)} = (\beta-\alpha)^{n(m-1)}.
\end{align}
\end{proof}

Since
$
 \deg D - 3 \deg D_2
  = 1092 - 3 \cdot 196
  = 504
  = \deg D_1,
$
one has $D = D_1 + 3 D_2$ and
\begin{align} \label{eq:disc}
 \Delta_T(t) = \frac{k(t)}{r(t)^3}.
\end{align}

\bibliographystyle{amsalpha}
\bibliography{bibs}

\end{document}